\documentclass[a4paper,10pt]{amsart}
\usepackage{amsmath,amsthm,xypic,
amssymb,mathrsfs,
pdfsync
}
\newtheorem{thm}{Theorem}[section]

\newtheorem{prop}[thm]{Proposition}
\newtheorem{cor}[thm]{Corollary}

\newtheorem{lem}[thm]{Lemma}
\theoremstyle{definition}

\newtheorem{defi}[thm]{Definition}

\newtheorem{rem}[thm]{Remark}
\newtheorem{rems}[thm]{Remarks}

\theoremstyle{plain}

\newcommand{\lla}{\langle\!\langle}
\newcommand{\rra}{\rangle\!\rangle}
\renewcommand{\k}{\mathop{\mathrm{Fix}}}

\renewcommand{\r}{\mathop{\mathrm{Res}}}

\renewcommand{\phi}{\varphi}

\newcommand{\alt}{\mathop{\mathrm{Alt}}}

\newcommand{\car}{\mathop{\mathrm{char}}}
\newcommand{\id}{\mathop{\mathrm{id}}}

\newcommand{\disc}{\mathop{\mathrm{disc}}}

\newcommand{\sym}{{\rm Sym}}

\author{\sc{A.-H. Nokhodkar}}
\date{}
\title{Applications of the Wall form to unipotent isometries of index two}

\begin{document}
\maketitle

\begin{abstract}
We investigate the Wall form of unipotent elements of index two in the orthogonal group and obtain a decomposition for these elements.
Also, in characteristic two, the relation between the Wall form and some invariants of the induced involution on the Clifford algebra is studied.
\end{abstract}

\section{Introduction}
The {\it Wall form} of an isometry is a nondegenerate bilinear form defined on its residual space (see \cite[pp. 12--13]{wall}).
According to \cite[Theorem 1.3.1]{wall2}, isometries of a quadratic space can be
classified, up to conjugation, by their Wall forms.
The Wall form was applied in \cite{hahn} and \cite{hahn2} in connection with the Zassenhaus decomposition \cite{zass} of isometries in characteristic $\neq2$.
It was also used in \cite{hahn3} to study the length of elements in
the commutator subgroup of the orthogonal group.
In characteristic $2$, this form was applied in \cite{con} to study the relation between the commutator subgroup of the orthogonal
group and the spinorial kernel of a nondegenerate defective quadratic space.

Let $(V,q)$ be a quadratic space over a field $F$.
An isometry $\tau$ of $(V,q)$ is called {\it unipotent} if $(\tau-\id)^k=0$ for some nonnegative integer $k$.
The minimum integer $k$ with this property is called the {\it unipotency index} of $\tau$.
According to \cite[(6.3)]{grove}, an isometry $\tau$ has unipotency index $2$ if and only if its residual space is totally singular.
Also, unipotent isometries of index $2$ are exactly those isometries whose Wall forms are antisymmetric (see \cite[p. 116]{hahn}).

In this work we investigate some applications of the Wall form to unipotent elements of index $2$ in the orthogonal group.
We start with some observations on {\it indecomposable} isometries (i.e., isometries which admit no invariant regular proper subspace).
In characteristic two, every unipotent isometry of index $2$ is an involution in the orthogonal group.
Such an involution is indecomposable if and only if it is either a $2$-dimensional reflection or a $4$-dimensional {\it basic null involution} (see \cite[Theorem 1]{wiitala}).
Some properties of $4$-dimensional basic null involutions, called {\it interchange isometries} are investigated in \cite{mahmoudi}.

In \S\ref{sec-int}, we generalize the definition of interchange isometries to arbitrary characteristic.
As we shall see in (\ref{defint}), a $4$-dimensional isometry is an inetrchange isometry if and only if it is indecomposable and unipotent of index $2$.
Also, using the Wall form we shall obtain in (\ref{char}) a decomposition of unipotent isometries of index $2$ into orthogonal sums of interchange isometries (and $2$-dimensional reflections if $\car F=2$).
In characteristic two, using the Jordan canonical form of a matrix,
this decomposition was already obtained in \cite[Theorem 2]{wiitala}.
However, the connection between
this decomposition and the Wall form has not been studied.

Let $(V,q)$ be a quadratic space over a field $F$.
Every involution $\tau$ in $O(V,q)$ induces a {\it natural} involution $J_\tau$, introduced by D. B. Shapiro \cite{shapiro}, on the Clifford algebra $C(q)$ satisfying $J_\tau(v)=\tau(v)$ for $v\in V$.
In \cite{mahmoudi4}, it was shown that every totally decomposable algebra with involution over a field of characteristic different from $2$ can be expressed as the Clifford algebra of a quadratic space with a natural involution induced by an involution in the orthogonal group (see \cite{mahmoudi} for a characteristic $2$ counterpart).
In characteristic $2$, it is readily seen that the converse is also true (see (\ref{tota})).
Thus, if $\car F=2$, there exists a natural correspondence between totally decomposable algebras with involution and
Clifford algebras with involution induced by an involution in the orthogonal group.

In \cite{dolphin3} and \cite{mahmoudi2} some invariants of a totally decomposable algebra with orthogonal involution $(A,\sigma)$ in characteristic $2$, such as the {\it Pfister invariant} and the subalgebra $\Phi(A,\sigma)$ were introduced.
In \cite[(6.5)]{mahmoudi2} it was shown that totally decom\-posable ortho\-gonal involutions on a given algebra can be classified by their Pfister invariants.
It was also shown that the Pfister invariant of $(A,\sigma)$ is isometric to an associative bilinear form defined naturally on $\Phi(A,\sigma)$ (see \cite[(5.5)]{mahmoudi2}).
In view of these characterizing properties, it would be interesting to study the relation between these invariants and the Wall form of $\tau$, when $(A,\sigma)$ is expressed as a Clifford algebra with natural involution $J_\tau$.
This investigation, which is the main object of \S\ref{sec-clif}, gives a geometric interpretation of the aforementioned invariants.
As a final application, we obtain in (\ref{totimes}), necessary and sufficient conditions for $(C(q),J_\tau)$ to be isomorphic to the matrix algebra with the transpose involution $(M_{2^n}(F),t)$.

\section{Preliminaries}
Let $V$ be a finite dimensional vector space over a field $F$.
A bilinear form $\mathfrak{b}$ on $V$ is called {\it alternating} if $\mathfrak{b}(u,u)=0$ for every $u\in V$.
Otherwise, $\mathfrak{b}$ is called {\it nonalternating}.
Every nonalternating symmetric bilinear space $(V,\mathfrak{b})$ has an {\it ortho\-gonal basis}, i.e., a basis $(v_1,\cdots,v_n)$ such that $\mathfrak{b}(v_i,v_j)=0$ for every $i\neq j$ (see \cite[(1.17)]{elman}).
In this case, $\mathfrak{b}$ is denoted by $\langle \alpha_1,\cdots,\alpha_n\rangle$, where $\alpha_i=\mathfrak{b}(v_i,v_i)\in F$.
The form $\lla\alpha_1,\cdots,\alpha_n\rra:=\langle1,\alpha_1\rangle\otimes\cdots\otimes\langle1,\alpha_n\rangle$ is called a {\it bilinear $n$-fold Pfister form}.
For a bilinear form $\mathfrak{b}$ we denote the form $\mathfrak{b}\perp\cdots\perp\mathfrak{b}$ ($r$ times) by $r\mathfrak{b}$.

Let $q:V\rightarrow F$ be a quadratic form.
 The bilinear form $\mathfrak{b}_q:V\times V\rightarrow F$ defined by
$\mathfrak{b}_q(u,v)=q(u+v)-q(u)-q(v)$
is called the {\it polar form} of $q$.
A subspace $W$ of $V$  is called {\it regular} if $W\cap  W^\perp=\{0\}$, where $W^\perp=\{x\in V\mid b_q(x,y)=0\ {\rm for\ all}\ y\in W\}$.
Otherwise, $W$ is called {\it singular}.
The pair $(V,q)$ is called a {\it quadratic space} if $V$ is regular.
A quadratic form $q$ is called {\it totally singular} if $\mathfrak{b}_q=0$.
Nontrivial totally singular quadratic forms only exist in characteristic two.
A nonzero vector $v\in V$ is called {\it isotropic} if $q(v)=0$.
A subspace $W$ of $V$ is called {\it totally isotropic} if all nonzero vectors of $W$ are isotropic.

Let $V$ be a $2n$-dimensional vector space over a field $F$ and let $\lambda=\pm1$.
A bilinear form $\mathfrak{b}$ (resp. a quadratic form $q$) on $V$ is called {\it hyperbolic} if it has a {\it hyperbolic basis}, i.e., a
basis $\mathcal{B}=(u_1,v_1,\cdots,u_n,v_n)$ such that: (i) $\mathfrak{b}(x,x)=0$ (resp. $q(x)=0$) for every $x\in\mathcal{B}$;
  (ii) $\mathfrak{b}(u_i,v_i)=\lambda\mathfrak{b}(v_i,u_i)=1$ (resp. $\mathfrak{b}_q(u_i,v_i)=1$) for every $i$; and (iii) all other pairs of
vectors in $\mathcal{B}$ are orthogonal.
The $2$-dimensional hyperbolic bilinear form is called the {\it hyperbolic plane} and is denoted by $\mathbb{H}_\lambda$.

The group of all isometries of a quadratic space $(V,q)$, i.e., the {\it orthogonal group} of $(V,q)$ is denoted by $O(V,q)$.
An isometry $\tau\in O(V,q)$ is called an {\it involution} if $\tau^2=\id$.
The {\it spinor norm} of an isometry $\tau$ is denoted by $\theta(\tau)$ (see \cite[p. 76 and p. 137]{grove}).
The {\it fixed space} and the {\it residual space} of an isometry $\tau$ are defined as~follows:
\begin{align*}
\k(\tau)=\{x\in V\mid \tau(x)=x\}\quad {\rm and} \quad
\r(\tau)=\{\tau(x)-x\mid x\in V\}.
\end{align*}
Note that
$\dim_F\r(\tau)+\dim_F\k(\tau)=\dim_FV$.
Also, it is readily seen that $\r(\tau)=\k(\tau)^\perp$.
The map
$\omega_\tau:\r(\tau)\times \r(\tau)\rightarrow F$
defined by
\[\omega_\tau(\tau(x)-x,\tau(y)-y)=\mathfrak{b}_q(\tau(x)-x,y), \quad {\rm for} \ x,y\in V,\] is called the  {\it Wall form} of $\tau$.
This map is a nondegenerate bilinear form satisfying $\omega_\tau(u,u)=-q(u)$ for $u\in\r(\tau)$ (see \cite[(1.1)]{hahn}).
Also, as stated in \cite[p. 116]{hahn}, this form is symmetric {\rm(}resp. antisymmetric{\rm)} if and only if $\tau^2=\id$ {\rm(}resp. $(\tau-\id)^2=0${\rm)}.

For $u\in V$ with $q(u)\neq0$, the involution $\tau_u\in O(V,q)$ defined by
$\tau_u(x)=x-\frac{\mathfrak{b}_q(u,x)}{q(u)}u$ is called the ({\it orthogonal}) {\it reflection} along $u$.
Also, if $x\in V$ is an isotropic vector and $w\in (Fx)^\perp$, the {\it Eichler transformation} $E_{x,w}:V\rightarrow V$ is~defined by
\[E_{x,w}(v)=v+\mathfrak{b}_q(v,x)w-\mathfrak{b}_q(v,w)x-q(w)\mathfrak{b}_q(v,x)x.\]
The map $E_{x,w}$ is a unipotent isometry of $(V,q)$ with $\r(E_{x,w})=Fx+Fw$.

Let $A$ be a central simple algebra over a field $F$.
An {\it involution} on $A$ is an antiautomorphism of $A$ of order $2$.
We say that $\sigma$ is of the {\it first kind} if $\sigma|_F=\id$.
Involutions of the first kind are either {\it symplectic} or {\it orthogonal} (see \cite[(2.5)]{knus}).
The {\it discriminant} of an orthogonal involution $\sigma$ is denoted by $\disc\sigma$ (see \cite[(7.2)]{knus}).
Also, the set of {\it alternating} and {\it symmetric} elements of $(A,\sigma)$ are defined as~follows:
\[\alt(A,\sigma)=\{a-\sigma(a)\mid a\in A\}\quad {\rm and}\quad\sym(A,\sigma)=\{a\in A\mid\sigma(a)=a\}.\]

\section{Decomposition in terms of the Wall form}\label{sec-int}

\begin{lem}\label{tauid}
For an isometry $\tau$ of a quadratic space $(V,q)$ the following conditions are equivalent:
$(1)$ $(\tau-\id)^2=0$.
$(2)$ $\r(\tau)\subseteq\k(\tau)$.
$(3)$ The restriction of $q$ to $\r(\tau)$ is totally singular.
\end{lem}

\begin{proof}
See \cite[pp. 46-47]{grove} (note that the argument given there works in arbitrary characteristic).
\end{proof}

\begin{lem}\label{v'}
Let $(V,q)$ be a quadratic space over a field $F$ and let $\tau$ be an isometry of $(V,q)$ with $(\tau-\id)^2=0$.
If $W$ is a complement of $\r(\tau)$ in $\k(\tau)$, then
\begin{itemize}
\item[$(1)$] $W$ is regular and $\tau=\id_W\perp\tau|_{W^\perp}$.
\item[$(2)$] $\dim_FW^\perp=2\dim_F\r(\tau)$.
\item[$(3)$] $\k(\tau|_{W^\perp})=\r(\tau|_{W^\perp})=\r(\tau)$.
\end{itemize}
\end{lem}

\begin{proof}
(1) Let $u\in W\cap W^\perp$ and let $x$ be an arbitrary vector in $\k(\tau)=W\oplus\r(\tau)$.
Write $x=w+v$ for some $w\in W$ and $v\in \r(\tau)$.
Since $u\in W^\perp$, we have $\mathfrak{b}_q(u,w)=0$.
Also, as $u\in W\subseteq\k(\tau)$ and $v\in\r(\tau)=\k(\tau)^\perp$, we get $\mathfrak{b}_q(u,v)=0$.
Thus, $\mathfrak{b}_q(u,x)=0$, i.e.,
$u\in\k(\tau)^\perp=\r(\tau)$.
The equality $W\cap \r(\tau)=\{0\}$ then implies that $u=0$, i.e., $W$ is regular.
By \cite[(7.22)]{elman} we have $V=W\perp W^\perp$, hence $\tau=\tau|_W\perp\tau|_{W^\perp}=\id_W\perp\tau|_{W^\perp}$.
\\
(2) Since $\dim_FW^\perp=\dim_FV-\dim_FW$, we get
\[\dim_FW^\perp=\dim_FV-(\dim_F\k(\tau)-\dim_F\r(\tau))=2\dim_F\r(\tau).\]
(3) Set $\tau'=\tau|_{W^\perp}$.
Then $\r(\tau)=\r(\id_W)\perp \r(\tau')=\r(\tau')$.
Using the second part, the relation $\dim_F\r(\tau')+\dim_F\k(\tau')=\dim_FW^\perp$
 yields $\dim_F\k(\tau')\penalty 0=\dim_F\r(\tau')$.
This, together with (\ref{tauid}) implies that $\k(\tau')=\r(\tau')$.
\end{proof}

\begin{defi}(\cite{wiitala})
Let $(V,q)$ be a $4$-dimensional quadratic space over a field $F$.
An isometry $\tau\in O(V,q)$ is called an {\it interchange isometry} if $\k(\tau)$ is a $2$-dimensional totally isotropic subspace of $(V,q)$.
\end{defi}

\begin{rems}\label{td}
(1) If $\tau$ is an interchange isometry, then $(V,q)$ is hyperbolic by \cite[(7.28)]{elman}.
 Also, dimension count implies that
$\k(\tau)=\k(\tau)^\perp=\r(\tau)$.
In particular, $(\tau-\id)^2=0$ by (\ref{tauid}).

\noindent (2) If $\car F\neq2$, then (\ref{tauid}) shows that for a $4$-dimensional isometry $\tau$ the following conditions are equivalent:
(i) $\tau$ is an interchange isometry.
(ii) $(\tau-\id)^2=0$ and $\tau\neq\id$. (iii) $\k(\tau)^\perp=\k(\tau)$.
\end{rems}

\begin{prop}\label{defint}
Let $(V,q)$ be a $4$-dimensional quadratic space over a field $F$.
For an isometry $\tau$ of $(V,q)$ the following conditions are equivalent:
\begin{itemize}
\item[$(1)$] $\tau$ is an interchange isometry.
\item[$(2)$] There exists a hyperbolic basis $(x,y,w,z)$ of $(V,q)$ such that
$(x,w)$ is a basis of $\k(\tau)$, $\tau(y)=y+w$ and $\tau(z)=z-x$.
\item[$(3)$] $(\tau-\id)^2=0$ and $\tau$ is indecomposable.
\end{itemize}
\end{prop}
\begin{proof}
$(1)\Rightarrow(2)$
By \cite[(3.5)]{mahmoudi}, it suffices to consider the case where $\car F\neq 2$.
Since $\r(\tau)$ is singular and $\dim_F\r(\tau)=2$, $\tau$ is an Eichler transformation by \cite[(3.10)]{knup}.
Write $\tau=E_{x,w}$ for some $x,w\in\r(\tau)\penalty 0=\k(\tau)$ and extend $(x,w)$ to a hyperbolic basis $(x,y,w,z)$ of $(V,q)$.
Then $\tau(y)=E_{x,w}(y)=y+w$ and $\tau(z)=E_{x,w}(z)=z-x$,
hence $(x,y,w,z)$ is the desired basis.

$(2)\Rightarrow(3)$
Since $\r(\tau)=Fx+Fw$ is totally isotropic, the equality $(\tau-\id)^2=0$ follows from (\ref{tauid}).
If $\car F\neq2$, then \cite[(247.1)]{snap} shows that $\tau$ is indecomposable.
Suppose that $\car F=2$.
Then every proper regular subspace $V_1$ of $V$ is $2$-dimensional over $F$.
If in addition $\tau(V_1)=V_1$, then $\tau|_{V_1}$ is either the identity on $V_1$ or a reflection (see \cite[Lem. 1]{wiitala}).
In both cases $\k(\tau|_{V_1})$ is regular, contradicting the fact that $\k(\tau|_{V_1})\subseteq\k(\tau)=Fx+Fw$ is totally isotropic.

$(3)\Rightarrow(1)$
Observe first that $\r(\tau)=\k(\tau)$, since otherwise by (\ref{v'})
every complement of $\r(\tau)$ in $\k(\tau)$ would be regular and $\tau$-invariant, contradicting the hypothesis.
Since $\dim_F\r(\tau)+\dim_F\k(\tau)=\dim_FV=4$, we obtain $\dim_F\k(\tau)=2$.
If $\car F\neq2$, then $\k(\tau)$ is totally isotropic by (\ref{tauid}), hence $\tau$ is an interchange isometry.
Suppose that $\car F=2$ and let $u$ be an arbitrary vector in $\k(\tau)=\r(\tau)$.
We claim that $q(u)=0$.
Choose $v\in V$ such that $u=\tau(v)-v$.
Then $Fu+Fv$ is a $\tau$-invariant subspace of $V$.
If $q(u)\neq0$, then $\mathfrak{b}_q(u,v)=\omega_\tau(u,u)=-q(u)\neq0$, i.e., $Fu+Fv$ is regular, a contradiction.
Thus, $q(u)=0$ and $\tau$ is an interchange isometry.
\end{proof}

\begin{rem}\label{g}
Let $(V,q)$ be a $4$-dimensional quadratic space over a field $F$ and let $\tau\in O(V,q)$ be an interchange isometry.
If $(x,y,w,z)$ is a basis of $V$ with the properties stated in (\ref{defint}), then a direct computation shows that $\tau$
coincides with the restriction to $V$ of the inner automorphism of $C(q)$ induced by $g:=1+wx\in C(q)$.
Furthermore, if $\car F=2$, the element $g$ can bo chosen more naturally as follows:
by \cite[(3.6)]{mahmoudi}, there exist $2$-dimensional regular subspaces $V_1$ and $V_2$ of $V$ such that $V=V_1\perp V_2$, $\tau(V_1)=V_2$ and $\tau(V_2)=V_1$.
 Let $(u_1,v_1)$ be any basis of $V_1$ with $\mathfrak{b}_q(u_1,v_1)=1$.
Then the inner automorphism of $C(q)$ induced by
 \[g=1+(u_1+\tau(u_1))(v_1+\tau(v_1))\in C(q),\]
restricts to $\tau$ on $V$.
The element $g$ corresponds to the {\it Goldman element} of $C(q|_{V_1})\otimes C(q|_{V_1})$ (see \cite[p. 35 and Exercise 8, p. 63]{knus}).
\end{rem}

\begin{rems}\label{rem}
Let $(V,q)$ be a quadratic space over a field $F$ and let $\tau$ be an isometry of $(V,q)$ with $(\tau-\id)^2=0$.
Consider the Wall form $\omega_\tau$.

\noindent(1) If $\car F\neq2$, then $\omega_\tau$ is an antisymmetric bilinear form, hence it is alternating.
By \cite[(1.8)]{elman}, $\omega_\tau$ is isometric to $n\mathbb{H}_{-1}$ for some $n$.

\noindent(2) If $\car F=2$, then $\tau$ is an involution in $O(V,q)$.
 Thus, $\omega_\tau$ is symmetric, which is either nonalternating or isometric to $n\mathbb{H}_{-1}$ for some $n$ (see \cite[(1.8)]{elman}).
\end{rems}

\begin{lem}\label{xinfix}
Let $(V,q)$ be a quadratic space over a field $F$ of characteristic $2$ and let $\tau\in O(V,q)$.
For a vector $u\in V$ the following conditions are equivalent:
\begin{itemize}
\item[$(1)$] There exists a regular subplane $V'$ of $V$ containing $u$ such that $\tau|_{V'}=\tau_u$.
\item[$(2)$] $u\in \r(\tau)$ and $\omega_\tau(u,u)\neq0$.
\end{itemize}
\end{lem}

\begin{proof}
$(1)\Rightarrow(2)$
Choose a vector $v\in V'$ such that $\mathfrak{b}_q(u,v)=-q(u)\in F^\times$.
Then $\tau(v)-v=u$, hence $u\in\r(\tau)$.
We also have $\omega_\tau(u,u)=-q(u)\neq0$.

 $(2)\Rightarrow(1)$ Write $u=\tau(v)-v$ for some $v\in V$ and set $V'=Fu+Fv$.
Then $\mathfrak{b}_q(u,v)=\omega_\tau(u,u)=-q(u)\neq0$, hence
$V'$ is regular.
Since $\tau(u)=u$ and $\tau(v)=v+u=v-\frac{\mathfrak{b}_q(u,v)}{q(u)}u$, we have $\tau|_{V'}=\tau_u$.
\end{proof}

\begin{lem}\label{int}
Let $(V,q)$ be a quadratic space over a field $F$ and let $\tau$ be an isometry of $(V,q)$ with $(\tau-\id)^2=0$.
For a $2$-dimensional subspace $U$ of $V$ the following conditions are equivalent:
\begin{itemize}
\item[$(1)$] There exists a $4$-dimensional regular subspace $V'$ of $V$ such that $\tau|_{V'}$ is an interchange isometry with $\r(\tau|_{V'})=U$.
\item[$(2)$] $U\subseteq\r(\tau)$ and $\omega_\tau|_U\simeq\mathbb{H}_{-1}$.
\end{itemize}
\end{lem}

\begin{proof}
$(1)\Rightarrow(2)$
Choose a hyperbolic basis $(x,y,w,z)$ of $V'$ with the properties stated in (\ref{defint}).
Then $U=Fx+Fw$, $\omega_\tau(w,w)=-q(w)=0$ and $\omega_\tau(x,x)=-q(x)=0$.
Also, the equalities $w=\tau(y)-y$ and $x=z-\tau(z)$ imply that $\omega_\tau(x,w)=\mathfrak{b}_q(x,y)=1$
and $\omega_\tau(w,x)=\mathfrak{b}_q(w,-z)=-1$, hence $\omega_\tau|_U\simeq\mathbb{H}_{-1}$.

$(2)\Rightarrow(1)$
Let $(x,w)$ be a hyperbolic basis of $U\subseteq\r(\tau)$.
Write $w=\tau(y)-y$ and $x=z-\tau(z)$ for some $y,z\in V$.
Set $V':=Fx+Fy+Fw+Fz$.
The relation $(\tau-\id)^2=0$ implies that $\tau(x)=x$ and $\tau(w)=w$, hence $V'$ is a $\tau$-invariant subspace of $V$ and
 $\k(\tau|_{V'})=\r(\tau|_{V'})\penalty 0 =U$.
By (\ref{tauid}) we have $\mathfrak{b}_q(x,w)=0$.
Also, $q(w)=-\omega_\tau(w,w)=0$ and $q(x)=-\omega_\tau(x,x)=0$, so $\k(\tau|_{V'})$ is a $2$-dimensional totally isotropic subspace of $V'$.
It remains to show that $V'$ is regular.
Let $u\in V'\cap {V'}^\perp$.
Since $\tau(V')=V'$ one has $\tau(u)-u\in V'\cap {V'}^\perp$.
On the other hand we have $\tau(u)-u\in\r(\tau|_{V'})=Fx+Fw$, hence there exists $a,b\in F$ such that $\tau(u)-u=ax+aw$.
Then
\[a=\omega_\tau(ax+bw,w)=\omega_\tau(\tau(u)-u,w)=\mathfrak{b}_q(\tau(u)-u,y)=0.\]
Similarly, we have $b=0$.
Thus, $\tau(u)=u$, i.e., $u\in\k(\tau|_{V'})=U=Fx+Fw$.
Repeating the above argument we obtain $u=0$, hence $V'$ is regular.
\end{proof}

\begin{thm}\label{char}
Let $(V,q)$ be a quadratic space over a field $F$ and let $\tau\in O(V,q)$ with $(\tau-\id)^2=0$.
Let $W$ be a complement of $\r(\tau)$ in $\k(\tau)$.
Then there exists a decomposition $V=W\perp V_1\perp\cdots\perp V_m$ into $\tau$-invariant regular subspaces, where $\tau|_W=\id$ and the integer $m$ and the subspaces $V_i$ can be described as follows:
\begin{itemize}
\item[$(1)$] If $\omega_\tau$ is alternating, then $m=\frac{1}{2}\dim_F\r(\tau)$, $\dim_FV_i=4$ and $\tau|_{V_i}$ is an interchange isometry for every $i$.
    More precisely, if $(x_1,w_1,\cdots,x_m,w_m)$ is a hyperbolic basis of $(\r(\tau),\omega_\tau)\simeq m\mathbb{H}_{-1}$, then
    the subspace $V_i$, $i=1,\cdots,m$, can be chosen such that $\r(\tau|_{V_i})=Fx_i+Fw_i$.
    \item[$(2)$] If $\omega_\tau$ is nonalternating, then $m=\dim_F\r(\tau)$, $\dim_FV_i=2$ and $\tau|_{V_i}$ is a reflection for every $i$.
More precisely, if $(u_1,\cdots,u_m)$ is an orthogonal basis of $(\r(\tau),\omega_\tau)$, then
the subspace $V_i$, $i=1,\cdots,m$, can be chosen such that $\tau|_{V_i}=\tau_{u_i}$.
\end{itemize}
\end{thm}

\begin{proof}
Set $s=\dim_F\r(\tau)$, $U=W^\perp$, $q'=q|_{U}$ and $\tau'=\tau|_{U}$.
By (\ref{v'}), $\tau=\id_W\perp\tau'$, $\r(\tau')=\r(\tau)$ and $\dim_FU=2s$.
We also have $(\tau'-\id)^2=0$ and $\omega_{\tau'}=\omega_{\tau}$.

(1)
Set $m=\frac{s}{2}$.
Then $\dim_FU=4m$.
  By (\ref{int}) there exists a $4$-dimensional subspace $V_1\subseteq U$ such that $\tau'|_{V_1}$ is an interchange isometry and $\r(\tau'|_{V_1})=Fx_1+Fw_1$.
Let $U'$ be the orthogonal complement of $V_1$ in $U$ with respect to $q'$.
Then $\dim_FU'=4(m-1)$.
For $i=2,\cdots,m$ and $v\in V_1$, we have $\mathfrak{b}_q(x_i,v)=\omega_\tau(x_i,\tau(v)-v)=0$, since $\tau(v)-v\in\r(\tau'|_{V' })=Fx_1+Fw_1$.
Similarly, we have $\mathfrak{b}_q(w_i,v)=0$, hence $x_i,w_i\in U'$.
It follows that $(x_2,w_2,\cdots,x_m,w_m)$ is a hyperbolic basis of $\omega_{\tau'|_{U'}}\simeq(m-1)\mathbb{H}_{-1}$.
By induction and dimension count one can write
\[U=V_1\perp\cdots\perp V_m,\]
 where every $V_i$ is $4$-dimensional subspaces of $U$ and $\tau'|_{V_i}$ is an interchange isometry with $\r(\tau'|_{V_1})=Fx_1+Fw_1$.
Thus, $V=W\perp U=W\perp V_1\perp\cdots\perp V_m$ is the desired decomposition.

(2) By (\ref{rem} (2)) we have $\car F=2$.
Set $m=s$.
In view of (\ref{xinfix}), there exists a $2$-dimensional subspace $V_1\subseteq U$ containing $U$ such that $\tau'|_{V_1}=\tau_{u_1}$.
Let $U'$ be the orthogonal complement of $V_1$ in $U$.
For $i=2,\cdots,n$, we have $u_i\in U'$, since $\mathfrak{b}_q(u_i,v_1)=\omega_\tau(u_i,u_1)=0$.
Thus, $(u_2,\cdots,u_n)$ is an orthogonal basis of $\omega_{\tau'|_{U'}}$.
The result again follows from induction and dimension count.
\end{proof}

\begin{rem}\label{wit}
If $\car F=2$, (\ref{char}) gives an alternative proof of the Wiitala decomposition of involutions in orthogonal group (see \cite[Theorem. 1]{wiitala}).
Also, the subspace $W$ in (\ref{char}) is a Wiitala subspace of $(V,\tau)$ in the sense of \cite[(3.8)]{mahmoudi}.
Finally, $\tau$ is of {\it interchanging kind} if and only if $\omega_\tau$ is alternating (see \cite[(3.19)]{mahmoudi}).
\end{rem}

\section{The induced involution on Clifford algebra in characteristic two}\label{sec-clif}
Throughout this section, $F$ denotes a field of characteristic $2$.

Let $(V,q)$ be a quadratic space over $F$.
Every isometry $\tau$ of $(V,q)$ with $(\tau-\id)^2=0$ is an involution in the orthogonal group.
Such an involution induces an involution $J_\tau$ on $C(q)$ satisfying $J_\tau(v)=\tau(v)$ for $v\in V$.
The aim of this section is to study the relation between the Wall form $\omega_\tau$ and some invariants of $(C(q),J_\tau)$, introduced in \cite{dolphin3} and \cite{mahmoudi2}.
We first recall some constructions from \cite{dolphin3} and \cite{mahmoudi2}.

An algebra with involution $(A,\sigma)$ is called {\it totally decomposable} if it decomposes as tensor products of {\it quaternion algebras} (i.e., central simple algebras of degree $2$) with involution.
Let $(A,\sigma)$ be a totally decomposable algebra of degree $2^n$ with orthogonal involution over $F$.
As proved in \cite[pp. 10-11]{mahmoudi2}, there exists a unique, up to isomorphism, subalgebra $S\subseteq \sym(A,\sigma)$ such that (i) $x^2\in F$ for $x\in S$;
(ii) $\dim_F S=\deg_FA$; (iii) $S$ is self-centralizing, i.e., $C_A(S)=S$; and (iv) $S$ is~generated as an $F$-algebra
by exactly $n$ elements.
The isomorphism class of $S$ is denoted by $\Phi(A,\sigma)$.
Furthermore, the algebra $\Phi(A,\sigma)$ has a {\it set of alternating generators}, i.e., a set $\{u_1,\cdots,u_n\}\subseteq\alt(A,\sigma)$
consisting of pairwise commutative square-central units satisfying
$u_{i_1}\cdots u_{i_l}\in\alt(A,\sigma)$ for every $1\leq l\leq n$ and $1\leq i_1<\cdots<i_l\leq n$.
If $\{u_1,\cdots,u_n\}$ is a set of alternating generators of $\Phi(A,\sigma)$, then $\Phi(A,\sigma)\simeq F[u_1,\cdots,u_n]$.

Let $(A,\sigma)\simeq\bigotimes_{i=1}^n(Q_i,\sigma_i)$ be a decomposition of $A$ into quaternion algebras with involution and let $\alpha_i\in F^\times$, $i=1,\cdots,n$, be a representative of the class $\disc\sigma_i\in F^\times/F^{\times2}$.
The form $\lla\alpha_1,\cdots,\alpha_n\rra$ is called the {\it Pfister invariant} of $(A,\sigma)$ and is denoted by $\mathfrak{Pf}(A,\sigma)$.
By \cite[(7.5)]{dolphin3}, $\mathfrak{Pf}(A,\sigma)$ is independent of the decomposition of $(A,\sigma)$.
Finally, if $\{u_1,\cdots,u_n\}$ is a set of alternating generators  of $\Phi(A,\sigma)$ with $u_i^2=\alpha_i\in F^\times$, $i=1,\cdots,n$, then $\mathfrak{Pf}(A,\sigma)\simeq\lla\alpha_1,\cdots,\alpha_n\rra$ (see \cite[pp. 11-12]{mahmoudi2}).

Our first result is the following restatement of \cite[(4.7)]{mahmoudi}, in view of (\ref{wit}):
\begin{lem}\label{res}
Let $(V,q)$ be a quadratic space over $F$ and let $\tau\in O(V,q)$ be an involution.
The involution $J_\tau$ on $C(q)$ is orthogonal if and only if $\r(\tau)=\k(\tau)$.
\end{lem}

\begin{lem}\label{sp}
Let $(V,q)$ be a $2$-dimensional quadratic space over $F$ and let $\tau\in O(V,q)$ be a reflection.
Then $(C(q),J_\tau)\simeq(M_2(F),t)$ if and only if $\theta(\tau)$ is trivial.
\end{lem}
\begin{proof}
If $(C(q),J_\tau)\simeq(M_2(F),t)$, then $\disc J_\tau$ is trivial, hence $\theta(\tau)=1$ by \cite[(4.11)]{mahmoudi}.
Conversely, if $\theta(\tau)$ is trivial, then $q$ represents $1$.
Thus, there exists $x\in V$ such that $x^2=1\in C(q)$.
Then $(x+1)^2=0$, hence $C(q)$ splits.
Also, \cite[(4.11)]{mahmoudi} shows that $\disc\sigma=\theta(\tau)$ is trivial, hence $(C(q),J_\tau)\simeq(M_2(F),t)$ by \cite[(7.4)]{knus}.
\end{proof}

The next result follows from \cite[(6.10)]{mahmoudi} and (\ref{sp}).
\begin{cor}\label{cint}
Let $(V,q)$ be a $4$-dimensional quadratic space over $F$ and let $\tau$ be an interchange isometry of $(V,q)$.
Then $(C(q),J_\tau)\penalty 0 \simeq(M_4(F),t)$.
\end{cor}

\begin{cor}\label{inter}
Let $(V,q)$ be a quadratic space over $F$ and let $\tau\in O(V,q)$ be an involution.
Let $W$ be a complement of $\r(\tau)$ in $\k(\tau)$.
If $\omega_\tau$ is alternating, then $(C(q),J_\tau)\simeq(C(q|_W),J_{\id})\otimes_F(M_{2^s}(F),t)$, where $s=\dim_F\r(\tau)$.
\end{cor}

\begin{proof}
The result follows from (\ref{char} (1)), \cite[(4.4)]{mahmoudi} and (\ref{cint}).
\end{proof}

\begin{lem}\label{tota}
Let $(V,q)$ be a quadratic space over $F$.
If $\tau$ is an involution in $(V,q)$, then $(C(q),J_\tau)$ is totally decomposable.
\end{lem}

\begin{proof}
By (\ref{char}), there exists a decomposition $V=W\perp V_1\perp\cdots\perp V_m$ into regular subspaces such that $\tau|_W=\id$ and either $\dim_FV_i=4$ and $\tau_i:=\tau|_{V_i}$ is an interchange isometry, or $\dim_FV_i=2$ and $\tau_i$ is a reflection.
By \cite[(4.4)]{mahmoudi} we have
\[(C(q),J_\tau)\simeq(C(q|_W),J_{\id})\otimes(C(q|_{V_1}),J_{\tau_1})\otimes\cdots\otimes(C(q|_{V_m}),J_{\tau_m}).\]
If $\tau_i$ is a reflection, then $(C(q|_{V_i}),J_{\tau_i})$ is a quaternion algebra with involution.
Thus, in view of (\ref{cint}), it suffices to show that $(C(q|_W),J_{\id})$ is totally decomposable.
Since $W$ is regular, one can write $W=W_1\perp\cdots\perp W_k$, where $k=\frac{1}{2}\dim_FW$ and every $W_i$ is a $2$-dimensional regular subspace of $W$.
Thus, $(C(q|_W),J_{\id})\simeq\bigotimes_{i=1}^k(C(q|_{W_i}),J_{\id})$ is totally decomposable.
\end{proof}

The next result states the relation between the subalgebra $\Phi(C(q),J_\tau)$ and the Wall form of $\tau$.
Recall that if $(V,\mathfrak{b})$ is a symmetric bilinear space over a field $F$,
the map $\phi_\mathfrak{b}:V\rightarrow F$ defined by $\phi_\mathfrak{b}(v)=\mathfrak{b}(v,v)$ is a quadratic form, called the {\it associated
quadratic form} of $\mathfrak{b}$.

\begin{prop}\label{cliff}
Let $(V,q)$ be a quadratic space over $F$ and let $\tau\in O(V,q)$ be an involution such that $\r(\tau)=\k(\tau)$.
If $\phi$ is the associated quadratic form of $\omega_\tau$, then $\Phi(C(q),J_\tau)\simeq C(\phi)\simeq C(q|_{\r(\tau)})$.
\end{prop}

\begin{proof}
Since $\phi(v)=\omega_\tau(v,v)=-q(v)$ for every $v\in \r(\tau)$, we have $\phi= -q|_{\r(\tau)}$.
Thus,
$C(q|_{\r(\tau)})\simeq C(\phi)$.
By (\ref{tauid}), $\phi$ is totally singular, hence $C(\phi)$ is commutative.
Let $(u_1,\cdots,u_s)$ be a basis of $\r(\tau)$ over $F$, where $s=\dim_F\r(\tau)=\frac{1}{2}\dim_FV$.
Since for every $i$, $u_i\in\sym(C(q),J_\tau)$ and $u_i^2\in F$, the commutativity of $C(\phi)$ implies that $C(\phi)\subseteq\sym(C(q),J_\tau)$ and $x^2\in F$ for every $x\in C(\phi)$.
Also, $\dim_FC(\phi)=2^s$ and $C(\phi)$ is generated as an $F$-algebra by $u_1,\cdots,u_s$, hence $C(\phi)$ is a Frobenius algebra by \cite[(3.4)]{mahmoudi2}.
Finally, as $\dim_FC(\phi)=2^s=\deg_FC(q)$, $C(\phi)\subseteq C(q)$ is self-centralizing by \cite[(2.2.3 (1))]{jacob}, thus $C(\phi)\simeq\Phi(C(q),J_\tau)$.
\end{proof}

\begin{thm}\label{clif}
Let $(V,q)$ be a quadratic space over $F$ and let $\tau\in O(V,q)$ be an involution such that $\r(\tau)=\k(\tau)$.
Let $s=\dim_F\r(\tau)=\frac{1}{2}\dim_FV$.
\begin{itemize}
\item[(1)] If $\omega_\tau$ is alternating then $(C(q),J_\tau)\simeq(M_{2^s}(F),t)$ and $\mathfrak{Pf}(C(q),J_\tau)\simeq\lla1,\cdots,1\rra$.
\item[(2)]
If $\omega_\tau$ is nonalternating then every orthogonal basis of $\r(\tau)$ is a set of alternating generators of $\Phi(C(q),J_\tau)$.
In particular, if $\omega_\tau\simeq\langle\alpha_1,\cdots,\alpha_s\rangle$ for $\alpha_1,\cdots,\alpha_s\in F^\times$ then $\mathfrak{Pf}(C(q),J_\tau)\simeq\lla\alpha_1,\cdots,\alpha_s\rra$.
\end{itemize}
\end{thm}

\begin{proof}
The first part follows from (\ref{inter}) and \cite[(5.7)]{mahmoudi2}.
To prove the second part let $(u_1,\cdots,u_s)$ be an orthogonal basis of $\r(\tau)$.
Then for $i=1,\cdots,s$, we have $u_i^2=q(u_i)\in F$ in $C(q)$.
The equality $\mathfrak{b}_q(u_i,u_j)=0$ implies that $u_iu_j=u_ju_i\in C(q)$ for every $i,j$.
By (\ref{char} (1)) one can write
$\tau=\tau_{u_1}\perp\cdots\perp\tau_{u_s}$, where every $\tau_{u_i}$ is a reflection of a $2$-dimensional $\tau$-invariant subspaces $V_i$ of $V$.
For $i=1,\cdots,s$, choose an element $v_i\in V_i$ such that $u_i=v_i-\tau(v_i)$.
Then
\[u_{i_1}\cdots u_{i_l}=v_{i_1}u_{i_2}\cdots u_{i_l}-J_\tau(v_{i_1}u_{i_2}\cdots u_{i_l})\in\alt(C(q),J_\tau),\]
for every $1\leq i_1<\cdots<i_l\leq s$ and $1\leq l\leq s$.
Since $\r(\tau)=\k(\tau)$, we have $\dim_FV=2s$, hence $\deg_FC(q)=2^s$
and $\{u_1,\cdots,u_s\}$ is a set of alternating generators of $\Phi(C(q),J_\tau)$.
\end{proof}

The proof of the following result is left to the reader.
\begin{lem}\label{totimes}
If $X$ is a symmetric matrix with $X^2\in F$, then $X^2\in F^2$.
\end{lem}

\begin{thm}
Let $(V,q)$ be a $2n$-dimensional quadratic space over $F$.
For an involution $\tau$ of $(V,q)$ with $\r(\tau)=\k(\tau)$ the following conditions are equivalent:
\begin{itemize}
\item[$(1)$] $(C(q),J_\tau)\simeq(M_{2^n}(F),t)$.
\item[$(2)$] $q(x)\in F^2$ for every $x\in\r(\tau)$.
\item[$(3)$] $\omega_\tau(x,x)\in F^2$ for every $x\in\r(\tau)$.
\item[$(4)$] Either $\omega_\tau\simeq n\mathbb{H}_{-1}$ or $\omega_\tau\simeq n\langle1\rangle$.
\end{itemize}
\end{thm}
\begin{proof}
Let $f:(C(q),J_\tau)\simeq(M_{2^n}(F),t)$ be an isomorphism of algebras with involution.
For every $x\in \r(\tau)$,
$f(x)$ is a symmetric matrix, since $\tau(x)=x$.
As $f(x)^2\penalty 0=q(x)\in F$, we get $f(x)^2\in F^2$ by (\ref{totimes}), i.e., $q(x)\in F^2$.
This proves $(1) \Rightarrow (2)$.

The implications $(2) \Rightarrow (3)$ and $(3) \Rightarrow (4)$ are evident.
To prove $(4) \Rightarrow (1)$ observe that in view of (\ref{inter}), it suffices to consider the case where $\omega_\tau\simeq n\langle1\rangle$.
Let $(u_1,\cdots,u_n)$ be an orthogonal basis of $\omega_\tau$.
By (\ref{char} (2)) one can write $V=V_1\perp\cdots\perp V_n$, where every $V_i$ is a $2$-dimensional subspace of $V$ containing $u_i$ and $\tau|_{V_i}=\tau_{u_i}$.
Thus, \cite[(4.4)]{mahmoudi} and (\ref{sp}) imply that
\[(C(q),J_\tau)\simeq\textstyle \bigotimes\limits_{i=1}^n(C(q|_{V_i}),J_{\tau_{u_i}})\simeq \bigotimes\limits_{i=1}^n(M_{2}(F),t)\simeq(M_{2^n}(F),t).\qedhere\]
\end{proof}

\scriptsize
A.-H. Nokhodkar, {\tt
     a.nokhodkar@kashanu.ac.ir}\\
Department of Pure Mathematics, Faculty of Science, University of Kashan, P.~O. Box 87317-53153, Kashan, Iran.
\end{document}